\documentclass[11pt,reqno,UTF8,twoside]{amsart}
\usepackage{mathrsfs}
\usepackage{amsfonts,amssymb,amsmath,amsthm}
\usepackage{cite}
\usepackage[colorlinks=true,citecolor=red]{hyperref}
\usepackage{titletoc}
\usepackage{geometry}
\usepackage{bm}
\usepackage{indentfirst}
\usepackage{graphicx}
\usepackage{float} 
\usepackage{booktabs}
\usepackage{longtable}
\usepackage{subfigure}
\usepackage{stmaryrd}
\usepackage{enumerate}
\usepackage{tikz}
\usepackage{ulem}
\usepackage{color,soul}
\usepackage{setspace}
\usepackage{stmaryrd}
\usepackage[pagewise]{lineno} 
\allowdisplaybreaks[2]
\usepackage{appendix}
\usepackage{cases}
\usepackage{todonotes}
\usepackage{enumitem}
\newcommand{\x}{\bm{x}}
\newcommand{\z}{\bm{z}}
\newcommand{\BT}{\bm{\theta}}
\newcommand{\he}{\mathfrak{e}}
\newcommand{\BV}{\bm{\varepsilon}}
\newcommand{\y}{\bm{y}}
\newcommand{\m}{\bm{m}}
\newcommand{\bk}{\bm{k}}

\newcommand{\e}{{\bm e}}
\newcommand{\R}{\mathbb R}

\newcommand{\Z}{\mathbb Z}

\newcommand{\C}{\mathbb C}

\newcommand{\q}{\bm{q}}

\newcommand{\A}{\mathcal{A}}
\newcommand{\Si}{\mathcal{S}}
\newcommand{\Ce}{\mathcal{C}}

\newcommand{\LL}{\mathbb{L}}

\newcommand{\N}{\mathbb{N}}
\newcommand{\T}{\mathbb{T}}

\newcommand{\sign}{{\rm sign}}

\newtheorem{thm}{Theorem}[section]

\newtheorem{lem}[thm]{Lemma}
\newtheorem{prop}[thm]{Proposition}

\newtheorem{rem}[thm]{\bf Remark}
\theoremstyle{definition}
\newtheorem{defn}[thm]{Definition}

\theoremstyle{statement}

\numberwithin{equation}{section}

\hypersetup{linkcolor=blue}

    \makeatletter
    \@namedef{subjclassname@2020}{\textup{2020} Mathematics Subject 
	Classification}
    \makeatother

\begin{document}
    
    \title[ Quantum ergodicity for Dirichlet operators ]
	{QUANTUM ERGODICITY FOR DIRICHLET-TRUNCATED\\ OPERATORS ON $\Z^d$}

    \author[H. Cao, S. Xiang]{ {\small H\MakeLowercase{ongyi} C\MakeLowercase{ao}, S\MakeLowercase{hengquan} X\MakeLowercase{iang} }}
    
    \address[Hongyi Cao]{School of Mathematical Sciences, Peking University, 100871, Beijing, China.}
    \email{chyyy@stu.pku.edu.cn}

    \address[Shengquan  Xiang]{School of Mathematical Sciences, Peking University, 100871, Beijing, China.}
    \email{shengquan.xiang@math.pku.edu.cn}

    \subjclass[]{}

		\keywords{Quantum ergodicity, Periodic Schr\"odinger operators on $\Z^d$,   Dirichlet condition, Eigenfunction correspondence}

    \begin{abstract}

 In this paper, we prove quantum ergodicity (a form of delocalization for eigenfunctions) for the Dirichlet truncations of the adjacency matrix on $\Z^d$. We also extend the result  to the cases of finite range observables and periodic Schr\"odinger operators with periods of length at most two.  This work partially answers a question asked by McKenzie and Sabri (Comm. Math. Phys. {\bf 403}(3), 1477--1509(2023) \cite{MS23}).      
    \end{abstract}

    \maketitle

\section{Introduction} 
 Serving as the quantum counterpart of classical ergodic theory, 	{\it quantum ergodicity}  refers to study the statistical behavior of quantum systems in the macroscopic regime and answer the  fundamental question of whether eigenfunctions of the Laplacian equidistribute. 
 Originated from the research in quantum mechanics by 
Einstein  \cite{Ein17}, it has now been extensively developed to mathematical objects such as manifolds and graphs, having  strong connections to spectral theory, quantum chaos, and random matrix theory; See for instance the excellent monograph by Anantharamann \cite[Chapter 1]{Ana22}.   

The  pioneering works of Shnirelman, Colin de Verdière, and Zelditch  \cite{Col85,Shn74,Zel87} first  established quantum ergodicity for  compact Riemannian manifolds: 
	
	\vspace{0.1cm}   
  \begin{flushleft}
        \textbf{Quantum Ergodicity Theorem.} Let $(M, g)$ be a compact Riemannian manifold. 
 
    Assume that the geodesic flow of $M$ is ergodic with respect to the Liouville measure. Let $\left(\psi_n\right)_{n \in \mathbb{N}}$ be an orthonormal basis of $L^2(M, g)$ consisting  of eigenfunctions of the Laplace-Beltrami operator $\Delta$:
  \end{flushleft}
	$$
	-\Delta \psi_n= \lambda_n \psi_n, \quad \lambda_n \leq \lambda_{n+1} \longrightarrow+\infty.
	$$
 Define $N(\lambda)=\#\left\{n:\ \lambda_n \leq \lambda\right\}$.	Let $a$ be a continuous function on $M$. 
   
    Then
	$$	\lim_{\lambda\to+\infty}\frac{1}{N(\lambda)} \sum_{n:  \lambda_n \leq \lambda}\left|
    \int_M a(x)|\psi_n|^2 d \operatorname{Vol}(x) - \int_M a(x) d \operatorname{Vol}(x)\right|^2 =0.
	$$

    \vspace{1mm}
    
    This theorem  implies that most of eigenfunctions of $\Delta$ (with  sequence density $1$), the 
	measures $\left|\psi_n(x)\right|^2 d\operatorname{Vol}(x)$ converges weakly to the uniform measure $d \operatorname{Vol}(x)$. In another word, most of eigenfunctions of Laplace equidistribute on $M$ in a weak sense. Later on, Rudnick and
	Sarnak \cite{RS94} made a stronger,  so-called \textit{quantum unique ergodicity} conjecture: for compact
	negatively curved manifolds, all Laplace eigenfunctions should become equidistributed in the high
	frequency limit. The conjecture is still widely open, but significant progress has been made in the last twenty
	years. We refer,  though not exclusively,  to the works \cite{Ana08,DJ18,DJN22} and references therein for more details.
\vspace{2mm}
    
The investigation of quantum ergodicity has been  extended to  {\it graphs setting}, offering novel insights to study  the spectral and dynamical delocalization properties of infinite graphs. In this setting, quantum ergodicity means that, for a sequence of finite graphs $G_N$ converging  (in the sense of Benjamini-Schramm) to some infinite graph $G$, most eigenfunctions   of the  adjacency matrix $\A_{G_N}$   equidistribute  spatially in a weak sense as the  size of graphs grows.
	
Quantum ergodicity for the adjacency matrix on large regular graphs that are spectral expanders with  {\it tree-like structures} was first established in the fundamental work  \cite{AL15}. See also   \cite{Ana17} for other three  proofs from different insights.   In  the  breakthrough work  \cite{AS19a}, this result   was further extended  to
Schr\"odinger operators on non-regular graphs assuming  absolutely continuous spectrum, which had many important applications such as Anderson model on regular tree \cite{AS17} and the adjacency matrix on finite cone type trees \cite{AS19b}.

Moreover, it was proved in \cite{AGS21} that the stronger statement, quantum unique ergodicity does not hold for the above deterministic sequences of regular graphs.  We  also  mention  the works of \cite{BY17,BYY20} and the references therein for quantum
unique ergodicity results in a different direction of random graphs setting.
\vspace{0.2cm}

The above results require the graphs  to have tree-like structures  (graphs with few cycles).	
In contrast, quantum ergodicity for  {\it periodic graphs}  remained largely unexplored until the recent work \cite{MS23}, which established (partial) quantum ergodicity for  periodic Schr\"odinger operators on graphs  with periodic  conditions. Here 
 the ``partial'' condition is necessary and  means that if the observable takes the same value in each periodic block, then  quantum ergodicity holds.   Specifically, quantum ergodicity holds when the fundamental cell has only  one vertex.  
\vspace{0.2cm}

Another important graph model is the lattice with Dirichlet conditions.  Quantum ergodicity  for this model remains  unknown so far, as asked by  McKenzie and  Sabri \cite[Remark 2.6]{MS23}: 
 \begin{center}
	{	\textit{Does  quantum ergodicity (resp. partial quantum ergodicity) hold  for adjacency matrix (resp. Schrödinger operator) with periodic conditions replaced by  Dirichlet conditions?}}
	\end{center}
Note that the boundary condition plays a crucial role in approximating the infinite graph as  even  a minor  perturbation of the operator may significantly change the eigenfunctions' behavior. For instance, the proof in \cite{MS23} relies on    periodic conditions to use Floquet transform.  However, this Floquet transform could not be extended to   Dirichlet conditions case.

		\vspace{0.05cm}

	\subsection{Main results}
This paper partially answers the question raised in  \cite{MS23} by establishing two results:
    \begin{itemize}
        \item[(1)]  Quantum ergodicity for the adjacency matrix on $\Z^d$  with Dirichlet conditions.
        \item[(2)]  Partial quantum ergodicity for Schrödinger operators with Dirichlet condition and period lengths at most  two. 
    \end{itemize} 
Note that the question of partial quantum ergodicity for Schrödinger operators with larger period lengths still remains open.

\vspace{2mm}
    
Throughout this paper we fix the dimension $d\in \N$. Denote, for  $N_1\leq N_2\in\R$ and  $N\in\N$, 
\begin{align*}
		[[N_1,N_2]]:=[N_1,N_2]\cap\Z \; \textrm{ and } \; 
		\Gamma_N:=[[1,N]]^d.
	\end{align*} 
We consider  the adjacency matrix on $\Gamma_N$, $\A_{\Gamma_N}:  \ell^2(\Gamma_N)\to\ell^2(\Gamma_N) $ given by 
	$$(\A_{\Gamma_N}\psi)(\x):=\sum_{\substack{\y\in \Gamma_N \\ \|\y-\x\|_1=1 }}\psi(\y), \ \ \psi\in \ell^2(\Gamma_N),\   \x\in \Gamma_N,  $$
	where $\|\cdot\|_{1}$ is the $l^1$-norm.

We also  call $\A_{\Gamma_N}$  the {\it Dirichlet-truncated adjacency matrix}, as it represents the finite truncation of  the adjacency matrix $\A_{\Z^d}$ restricted to the lattice cube  $\Gamma_N$ under zero boundary condition.  
Another natural  boundary condition is the periodic condition, which has been well-studied in \cite{MS23}, see also Section \ref{est} for  definitions and statements of the results.  Notably, there are significant differences between the two conditions, which will be discussed in Section \ref{Diff}.
    
	Our first result  establishes quantum ergodicity for  $\A_{\Gamma_N}$.
	\begin{thm}\label{thm}
		Let $\{\psi_N^{(j)}\}_{j=1}^{N^d}$ be an orthonormal basis of $\ell^2(\Gamma_N)$ consisting of eigenfunctions of $\A_{\Gamma_N}$ and let  $a_N: \Gamma_N\times \Gamma_N\to \C$ be a  sequence of  diagonal  observables  satifying  $\|a_N\|_\infty:=\sup_{\x\in\Gamma_N}|a_N(\x,\x)|\leq 1$  for any $N\in \N$. Then we have 	 
		\begin{equation}\label{1}
			\lim _{N \rightarrow \infty} \frac{1}{N^d} \sum_{j=1}^{N^d}\left| \langle \psi_N^{(j)}, a_N \psi_N^{(j)}\rangle-\langle a_N\rangle\right|^2=0
		\end{equation}
		where  $ \langle\cdot,\cdot \rangle$ is the inner product so that  $ \langle \psi_N^{(j)}, a_N \psi_N^{(j)}\rangle=\sum_{\x \in \Gamma_N}a_N(\x,\x) | \psi_N^{(j)}(\x)|^2 $ and $\langle a_N\rangle:=\frac{1}{\#\Gamma_{N}} \sum_{\x \in \Gamma_N} a_N(\x,\x)$ is the uniform average. 
	\end{thm}
 Theorem \ref{thm} says that for large $N$ and for most $j$, in a weak sense  by testing
a function $a_N$ on $\Gamma_N$,  the  probability measure 
	$\sum_{\x \in \Gamma_N}| \psi_N^{(j)}(\x)|^2\delta_{\x}$ is ``close'' to the uniform measure 	$\frac{1}{\#\Gamma_{N}}\sum_{\x \in \Gamma_N} \delta_{\x}$.
    
\vspace{2mm}

We can generalize Theorem \ref{thm} by replacing  diagonal observables  $a_N$ with finite range observables  $K_N$. Before presenting the theorem, we introduce some definitions.

For an observable $K: \Gamma_N\times \Gamma_N\to \C$ and a function $\psi\in \ell^2(\Gamma_N)$, we define  
\begin{equation}\label{dco}
	\langle K\rangle_{\psi}:=   \sum_{\z } \sum_{\x \in L_{\z} }  \frac{1}{\#L_{\z}} K(\x,\x+\z) \langle\psi,\rho_{\z}\psi\rangle,
\end{equation}
where  \begin{align}
\label{Lz}
	L_{\z}:= \{\x\in \Gamma_N:\ \x+\z\in \Gamma_N\}, \quad 
(\rho_{\z}\psi)(\x):=\begin{cases}
	\psi(\x+\z), \ \ &\x+\z\in\Gamma_{N}, \\
	0, \ \ &\x+\z\notin\Gamma_{N}.
\end{cases}
\end{align}

	\begin{thm}\label{thm.}
	Let $\{\psi_N^{(j)}\}_{j=1}^{N^d}$ be an orthonormal basis of $\ell^2(\Gamma_N)$ consisting of eigenfunctions of $\A_{\Gamma_N}$. Fixing $R\in \N$, let   $K_N$ be a sequence of finite range  observables  satisfying
	\begin{itemize}
		\item[\tiny$\bullet$] $K_N(\x,\y)=0$  for any $N\in \N$ and $\|\x-\y\|_1>R,$
		\item[\tiny$\bullet$] $|K_N(\x,\y)|\leq 1$ for any $N\in \N$ and $\x,\y\in \Gamma_{N}.$
	\end{itemize}  	 Then we have  	\begin{equation}\label{1.}
		\lim _{N \rightarrow \infty} \frac{1}{N^d} \sum_{j=1}^{N^d}\left| \langle \psi_N^{(j)}, K_N \psi_N^{(j)}\rangle-\langle K_N\rangle_{\psi_N^{(j)}}\right|^2=0.
	\end{equation} 
In particular, taking $R=0$ in the statement above gives back Theorem \ref{thm}.
\end{thm}
	Since $K_N$ is	supported at distance $\leq R$ from the diagonal,  the inner product
	 $	\langle \psi_N^{(j)}, K_N \psi_N^{(j)}\rangle$ is  $$  \sum_{\|\z\|_1\leq R }\  \sum_{\x \in L_{\z} }K_N(\x,\x+\z) \overline{\psi_N^{(j)}(\x)}(\rho_{\z}\psi_N^{(j)})(\x).$$
	 Thus, Theorem \ref{thm.}   says that for large $N$, for most $j$ and  for any  $\|\z\|_1\leq R$,  in a weak sense  by testing
	 a function on the  set   $L_{\z}$, the  signed measure
	 $$\sum_{\x \in L_{\z}}\overline{\psi_N^{(j)}(\x)}(\rho_{\z}\psi_N^{(j)})(\x)\delta_{\x}$$ is ``close'' to the dilated uniform measure 	$$ \langle\psi_N^{(j)},\rho_{\z}\psi_N^{(j)}\rangle\cdot \frac{1}{\#L_{\z}}\sum_{\x \in L_{\z}} \delta_{\x} .$$
\begin{rem}
The  value $\langle K\rangle_{\psi}$ in \eqref{1.} is termed   the  \textit{eigenfunction correlator}, a quantity previously investigated in regular graphs \cite{Ana17} and periodic graphs \cite{MS23}. Interestingly, we find a distinctive property (c.f., Proposition \ref{pwucha}) that holds for one-dimensional Dirichlet conditions operators, yet fails in the aforementioned two cases. Detailed  discussions are provided  in Section~ \ref{FEC}.
\end{rem}
	Next, we consider the  {\it Dirichlet-truncated periodic Schr\"odinger operators}. 

	Let $\bm e_l$ be the normalized  vector that takes the value  $1$ at the $l$-th site. Given $\q=(q_1,\cdots,q_d)\in \N^d$, let $\LL$ be a lattice of full rank with a basis $q_1\e_1,\cdots,q_d\e_d$:
	$$\LL=\{n_1q_1\e_1+\cdots+n_dq_d\e_d:\ n_1,\cdots,n_d\in\Z\}.$$
	Let $H$ be the  Schr\"odinger operator on $\ell^2(\mathbb{Z}^d)$ with potentials $V$ periodic with respect to $\LL$:
	\begin{equation}\label{pz}
		(H \psi)(\x)=(\A_{\Z^d} \psi)(\x)+V(\x) \psi(\x):=\sum_{\substack{\y\in \Z^d \\ \|\y-\x\|_1=1 }}\psi(\y)+V(\x) \psi(\x), 
	\end{equation}
	where 
	$V(\x+\y)=V(\x)$ for any $\y \in \LL.$
    
	Define 
	$$\Lambda_N:=\prod_{l=1}^d[[1,q_lN]].$$
	In particular, we call $\Lambda_1=\prod_{l=1}^d[[1,q_l]]$ the fundamental block.  
    
	We denote by $\mathcal{H}_{\Lambda_N}$ the Dirichlet truncation of $H$ to $\Lambda_{N}$:
	\begin{equation}\label{dr}
		(\mathcal{H}_{\Lambda_N} \psi)(\x):=(\A_{\Lambda_N}\psi)(\x)+V(\x) \psi(\x), \ \ \psi\in \ell^2(\Lambda_N),\   \x\in \Lambda_N,
	\end{equation}
	where 
	$$(\A_{\Lambda_N}\psi)(\x):=\sum_{\substack{\y\in \Lambda_N \\ \|\y-\x\|_1=1 }}\psi(\y).$$
	It is well-known that the spectrum of  \eqref{pz} is purely absolutely continuous (c.f., \cite{Kuc16}), which is also called the   spectral delocalization in the literature of discrete Schr\"odinger operators. Thus it is natural to ask: Are there any   ergodicity/delocalization properties on the spatial  behavior of the eigenfunctions of $\mathcal{H}_{\Lambda_N}$  as $\Lambda_N$ converges to $\Z^d$?  
	
	Note that  the  eigenfunctions of $\mathcal{H}_{\Lambda_N}$ may distribute sensitively depending on the potential.
	\begin{prop}\label{app}
	Generally, quantum ergodicity does not hold  for $\mathcal{H}_{\Lambda_N}$.
	\end{prop}
	 However, we can expect the partial quantum ergodicity, in the sense that the  $\ell^2$-mass
	of most  eigenfunctions of $\mathcal{H}_{\Lambda_N}$  over every periodic block $\Lambda_1+\bk$ $(\bk\in \mathbb{L})$ is approximately the same so that if the observable takes the same value in each periodic block, then  quantum ergodicity holds. 	
	
	Partial quantum ergodicity for periodic Schr\"odinger  operators on $\Z^d$ with periodic conditions was  proved by \cite{MS23} combined with  \cite{Liu24}, where \cite{MS23}  established a general criteria concerning Floquet eigenvalues  for (partial) quantum ergodicity on periodic graphs, and in the case of $\Z^d$ lattice,  this criteria was verified by  \cite{Liu24} through  delicate analysis of Bloch varieties. 
	
In  the case of Dirichlet conditions, we obtain:    
	\begin{thm}\label{ps}
		Assume that  \begin{equation}\label{cd}
			q_l\in \{1,2\} \ \  \text{{\rm for} $1\leq l\leq d$.}
		\end{equation} Let $\mathcal{H}_{\Lambda_N}$ be the operator defined by \eqref{dr} and $\{\psi_N^{(j)}\}_{j=1}^{\#\Lambda_N}$ be an orthonormal basis of $\ell^2(\Lambda_N)$ consisting of eigenfunctions of $\mathcal{H}_{\Lambda_N}$. 
		Assume  the sequence of  diagonal  observables $a_N: \Lambda_N\times \Lambda_N\to \C$ satisfies: 
		\begin{itemize}
        [leftmargin=2em]
			\item[\tiny$\bullet$] $\|a_N\|_{\infty}\leq 1$ for any $N\in \N$.
			\item[\tiny$\bullet$] For all $N\in \N$  and  $\x,\y\in \Lambda_1$, 	\begin{equation}\label{lc}
				\sum_{\bk\in \mathbb{L}_N}a_N(\x+\bk,\x+\bk)=\sum_{\bk\in \mathbb{L}_N}a_N(\y+\bk,\y+\bk),
			\end{equation} where $\LL_N=\{n_1q_1\e_1+\cdots+n_dq_d\e_d:\ n_1,\cdots,n_d\in[[0,N-1]]\}.$
		In particular, \eqref{lc} holds if $a_N$ is locally constant, in the sense that it takes a constant value on each periodic
			block: $a_N(\x+\bk,\x+\bk)=a_N(\y+\bk,\y+\bk)$ for any $ \x,\y\in \Lambda_1,\bk\in \mathbb{L}_N$.
		\end{itemize}  Then we have 	 
		\begin{equation*}\label{11}
			\lim _{N \rightarrow \infty} \frac{1}{\#\Lambda_{N}} \sum_{j=1}^{\#\Lambda_{N}}\left| \langle \psi_N^{(j)}, a_N \psi_N^{(j)}\rangle-\langle a_N\rangle\right|^2=0,
		\end{equation*}
		where  $\langle a_N\rangle$ is the uniform average of $a_N$ on $\Lambda_{N}$.
	\end{thm}
    Theorem \ref{ps} is restricted to the  case of period lengths at most two (assumption \eqref{cd}) due to technical reasons.  For a detailed explanation and related further questions, we refer  to Section  \ref{zh}.

    \subsection{Comparison between periodic   conditions  and Dirichlet   conditions}\label{Diff}
The analysis for periodic conditions was  well-established in \cite{MS23}. Here we outline the key  differences between the two boundary conditions, with  emphasis on  the difficulties inherent in  Dirichlet conditions:
\begin{itemize}[leftmargin=2em]
    \item[\tiny$\bullet$] \textit{Graph structure:}
From a graph-theoretic perspective,  periodic boundary conditions correspond to cycle graphs (or $d$-dimensional  toroidal graphs), while Dirichlet  conditions correspond to finite lattice graphs (or cube graphs with zero boundaries). The former exhibit stronger homogeneity,  as the  $r$-balls  (i.e., $B_G(\x,r):=\{\y\in G:\ {\rm d}(\x,\y)\leq r\}$)     in such a graph are isomorphic for  all vertices $\x$. It seems plausible that
this homogeneity supports the  equidistribution  of eigenfunctions.
\item[\tiny$\bullet$] \textit{Eigenbasis:}
Periodic conditions admit a canonically equidistributed eigenbasis, namely Bloch states (e.g., \cite[Section 5.2]{MS23}) and the question  lies  in addressing  the multiplicity of eigenvalues and establishing control for all possible eigenbases. While  Dirichlet  conditions lack such a privileged basis. Even for the adjacency matrix on finite lattices, quantum ergodicy is highly non-trivial as  the eigenbasis comprises sine functions, which seem incompatible with equidistribution. 
\item[\tiny$\bullet$] \textit{Methodology:} For periodic conditions, \cite{MS23}   employed Floquet transforms to establish diagonaliztion  and  Floquet eigenvalues to analyze eigenvalue multiplicity. While as noted in  \cite[Remark 2.6]{MS23}, these  techniques  could not be extended to   Dirichlet conditions and a natural
candidate for Dirichlet conditions is the  sine transform. However, the algebraic structure of sine basis lacks the well-behaved translational properties of exponentials (Bloch states) used in periodic conditions, obstructing  direct generalization of the quantum ergodicity proof. This problem requires another  framework to address.
\end{itemize}

	\subsection{Main ideas and new ingredients} We provide  two different proofs for Theorem \ref{thm} and extend the second one to prove Theorem \ref{thm.} and \ref{ps}. The two proofs  can be read independent of each other.  Let us  state the main ideas and new ingredients  of the two proofs respectively.  
		\subsubsection{Explicit matrix calculations and Fourier coefficient estimation}\label{sec:fourstepst} The first proof follows  a by-now standard four-step strategy for quantum ergodicity, where we handle the main difficulty of quantum variance via explicit matrix calculations and Fourier coefficient estimation.  
	
	The proof of quantum ergodicity  has been  largely developed  to a four-step strategy, which is discussed in Anantharaman's monograph \cite[Section 2.5 and 6.2]{Ana22}. In the framework of graphs, it  can be summarized as: 
	\begin{itemize}[leftmargin=2em]
		\item[\tiny$\bullet$]\textit{First step:} Introduce an appropriate ``quantum variance'' $\operatorname{Var}(a_N)$ for the  observable $a_N$ so that the question of quantum ergodicity turns into the estimation of the limit of quantum variance.
		\item[\tiny$\bullet$]\textit{Second step:}  Use the eigenfunction property to establish  invariance  of  $\operatorname{Var}(a_N)$  under a certain ``quantum dynamics''. This transforms it  into a better  form $\operatorname{Var}({a'_N})$, which is  much easier to handle.
		\item[\tiny$\bullet$]\textit{Third step:} Find an upper bound on $\operatorname{Var}({a'_N})$ in terms of the $\ell^2$-norm  of ${a'_N}$ in some appropriate Hilbert space $\mathcal H$.
		\item[\tiny$\bullet$]\textit{Forth step:} Use  the delocalization/ergodicity structures on the graphs to  obtain the  bound $\|{a'_N}\|_{\mathcal{H}}= o(\|a_N\|_\infty)_{N\to\infty}$, thus one has  $\operatorname{Var}(a_N)\to0$ as $N\to\infty$, which finishes the proof of quantum ergodicity.
	\end{itemize}
	
When applied to specific graph models, each of the four steps requires unique analysis. Generally speaking,  Steps $1$--$2$  need observational examination of graph structures, whereas  Steps $3$--$4$ require more technical analysis.
	
	In the current framework,  we consider the standard  quantum variance in Step $1$:	$$\operatorname{Var} (a_N):= \frac{1}{N^d} \sum_{j=1}^{N^d}\left|\langle \psi_N^{(j)}, a_N \psi_N^{(j)}\rangle\right|^2.$$

For such an appropriate quantum dynamics in Step 2, we draw motivation from  \cite{KK05,MS23}, where  \cite{KK05} studied the ergodic average of  Laplacian on torus, and \cite{MS23} extended this framework to periodic graphs.
 We find that this  ergodic average argument also applies to  the Dirichlet conditions. 
  Specifically, we will  analyze the quantum dynamics given by the   limit time-averaged dynamics of the observable:
\begin{equation}\label{..}
	\lim_{T\to \infty}\frac{1}{T} \int_0^T e^{-{\rm i}t\A_{\Gamma_N}}a_Ne^{{\rm i}t\A_{\Gamma_N}}dt.
\end{equation}

 In Step $3$, by  a simple computation (c.f., \eqref{4}),  we show that the  quantum variance can be bounded via the Hilbert-Schmidt norm of the matrix.  
 
 Now we arrive at Step 4.  In the periodic conditions case,  \cite{MS23} introduced  the  Floquet transform  and  ``quantization'' argument for the proof.  
  However, since the Floquet transform requires strictly periodic conditions, it could not be extended  to the  Dirichlet conditions. Instead, we consider a Dirichlet-adapted sine basis (c.f., \eqref{sb}) for a decomposition of  $\A_{\Gamma_N}$. Through this specialized basis, we perform explicit matrix calculations to derive the  explicit expression of \eqref{..}, which is denoted by  $a_N^\infty$,  formulated via the  discrete Fourier coefficients of $a_N$ (c.f., \eqref{12} and \eqref{fj}). Now the issue becomes the estimation of  these Fourier coefficients. Through a further analysis concerning the  multiplicity of  eigenvalues of $\A_{\Gamma_N}$ (c.f., Lemma \ref{c1}), we prove  that the  non-zero frequency terms are bounded by $N^{d-1}$.
  Here, Lemma \ref{c1} can be interpreted as the Dirichlet conditions counterpart to the Floquet eigenvalue assumption (originally introduced for periodic conditions) in \cite[equation (1.3)]{MS23}.
     By partitioning  Fourier coefficients into orthogonal classes and applying Bessel's inequality to each class, we establish $ \|a_N^\infty-\langle a_N\rangle\operatorname{Id}_{\Gamma_N}\|^2_{\rm HS}=O(N^{d-1})$, thereby proving the convergence of quantum variance to zero. 

	We believe that explicit matrix calculations provide an intuitive framework for addressing the technical difficulties inherent in sine basis representations. Moreover, this approach also applies to the periodic conditions case, provided an explicit basis is known.

	\subsubsection{Eigenfunction correspondence}  
	The second proof introduce an elegant eigenfunction correspondence to relate the Dirichlet conditions problems to established results in the periodic conditions case. While non-self-contained, this approach offers greater conciseness.  
	
	We note that the eigenvalues of the Dirichlet condition operator $\mathcal{A}_{\Gamma_N}$, given by $\sum_{l=1}^d 2\cos\frac{k_l\pi}{N+1}$ with $\bk \in \Gamma_N$, form a subset of the eigenvalues of the periodic condition operator $\mathfrak{A}_{\Gamma_{2N+2}}$ on an extended lattice cube $\Gamma_{2N+2}$, whose eigenvalues are $\sum_{l=1}^d 2\cos\frac{k_l\pi}{N+1}$ with $\bk \in [[0,N]]^d$. Motivated by this observation,  through zero-extension and symmetric reflection of eigenfunctions, we establish an injective map $\he$ from $\mathcal{A}_{\Gamma_N}$-eigenfunctions to $\mathfrak{A}_{\Gamma_{2N+2}}$-eigenfunctions (c.f., Proposition \ref{eb}). This eigenfunction correspondence enables the transfer of eigenfunction properties  (particularly quantum ergodicity results) from the known periodic conditions case to the Dirichlet one.  
	
	Generally, periodic conditions operators permit more tractable spectral analysis through Bloch-Floquet theory  compared to Dirichlet conditions operators. The eigenfunction correspondence establishes a connection between the two operator classes, suggesting further applications to eigenvalue/eigenvector problems besides quantum ergodicity. Notably, the framework admits natural extension to finite range observables and  periodic Schr\"odinger operators with periods of length at most two.  
	
	\vspace{0.1cm}
	
	\textbf{Organization of the paper:}  We give the first proof of  Theorem \ref{thm} (proof via the four-step strategy) in Section \ref{pf1}. Then we state the established periodic conditions results for use in the second proof  in Section \ref{est} and introduce the eigenfunction correspondence  in Section \ref{map}.   The second proof of Theorem \ref{thm} (proof using eigenfunction correspondence) is given in Section \ref{proof1}.  Finally, we  extend the eigenfunction correspondence to finite range observables in Section \ref{pf21} and to  periodic Schr\"odinger operators in  \ref{pf22}. Further discussions for eigenfunction correlators and further questions are included in Section \ref{fur}.

	\section{Quantum ergodicity for the adjacency matrix on $\Z^d$  with Dirichlet conditions}\label{pf1}
    This section is devoted to the proof of quantum ergodicity for the adjacency matrix on $\Z^d$  with Dirichlet boundary conditions, as stated in  Theorem \ref{thm}. As illustrated in Section \ref{sec:fourstepst}, we follow the four-step strategy, with a particular emphasis on explicit matrix calculations and Fourier coefficient estimation. An alternative proof of the theorem is also presented in Section~ \ref{pf2}. 
	\begin{flushleft}
		\textbf{Step 1:} Introduce the quantum variance.
	\end{flushleft}   
	
	For an observable $T_N: \Gamma_N\times \Gamma_N\to \C$, given an  orthonormal eigenfunction  basis $\{\psi_N^{(j)}\}_{j=1}^{N^d}$ of $\A_{\Gamma_N}$,   we define the quantum variance for $T_N$:
	$$\operatorname{Var} (T_N):= \frac{1}{N^d} \sum_{j=1}^{N^d}\left|\langle \psi_N^{(j)}, T_N \psi_N^{(j)}\rangle\right|^2.$$
	Then \eqref{1} is equivalent to 
	\begin{equation}\label{2}
		\lim _{N \rightarrow \infty}  \operatorname{Var}(a_N-\langle a_N\rangle\operatorname{Id}_{\Gamma_N})=0.
	\end{equation}	
	\begin{flushleft}
		\textbf{Step 2:} Use the eigenfunction property to note that the quantum variance is invariant under ``time-averaged dynamics''. 
	\end{flushleft}
	
	Consider  the   time-averaged dynamics  for $T>0$
	$$a^T_N:=\frac{1}{T} \int_0^T e^{-{\rm i}t\A_{\Gamma_N}}a_Ne^{{\rm i}t\A_{\Gamma_N}}dt .$$ 
	We denote $a^{\infty}_N=\lim_{T\to \infty}a^T_N$. (We will show that the limit always exists.)
	
	Denote by $\lambda_N^{(j)}$  the corresponding  eigenvalue of $\psi_N^{(j)}$. Since for any $t>0$, 
	\begin{align*}
		\langle \psi_N^{(j)}, a_N \psi_N^{(j)}\rangle&=\langle e^{{\rm i}t\lambda_N^{(j)}} \psi_N^{(j)}, a_N e^{{\rm i}t\lambda_N^{(j)}} \psi_N^{(j)}\rangle \\
		&= \langle e^{{\rm i}t\A_{\Gamma_N}} \psi_N^{(j)}, a_N e^{{\rm i}t\A_{\Gamma_N}} \psi_N^{(j)}\rangle \\
		&=\langle  \psi_N^{(j)}, e^{-{\rm i}t\A_{\Gamma_N}} a_N e^{{\rm i}t\A_{\Gamma_N}} \psi_N^{(j)}\rangle, 
	\end{align*}
	averaging the above equation for  $t$ on $[0,T]$   and letting $T\to \infty$  yield 
	$$
	\langle \psi_N^{(j)}, a_N \psi_N^{(j)}\rangle= \langle \psi_N^{(j)}, a_N^T \psi_N^{(j)}\rangle=
	\langle \psi_N^{(j)}, a_N^{\infty} \psi_N^{(j)}\rangle.$$
It follows from the above equation that   \begin{equation}\label{3}
		\operatorname{Var}(a_N)=\operatorname{Var}(a_N^{\infty}).
	\end{equation}
	\begin{flushleft}
		\textbf{Step 3:} Bound the quantum variance by the Hilbert-Schmidt norm.
	\end{flushleft}	
	
	Since $\{\psi_N^{(j)}\}_{j=1}^{N^d}$ is an orthonormal basis,  for any observable $T_N$, we have 
	\begin{align}
		\nonumber	\operatorname{Var} (T_N)&\leq\frac{1}{N^d}  \sum_{j=1}^{N^d}   \|\psi_N^{(j)}\|^2\|T_N \psi_N^{(j)}\|^2\\
		\nonumber	&\leq \frac{1}{N^d}\sum_{j=1}^{N^d} \langle T_N \psi_N^{(j)}, T_N \psi_N^{(j)}\rangle\\
		\nonumber	&=\frac{1}{N^d}\operatorname{Trace} (T_N^*T_N)\\
		&=\frac{1}{N^d}\|T_N\|_{\rm HS}^2, \label{4}
	\end{align}
	where $\|T_N\|_{\rm HS}^2=\sum_{\x,\y\in \Gamma_N}|T_N(\x,\y)|^2.$ 
	Combing \eqref{3} and \eqref{4}, in order to prove \eqref{2}, we only need to prove 
	\begin{equation}\label{=}
		\lim _{N \rightarrow \infty} \frac{1}{N^d} \|a_N^\infty-\langle a_N\rangle\operatorname{Id}_{\Gamma_N}\|^2_{\rm HS}=0.
	\end{equation}
	\begin{flushleft}
		\textbf{Step 4:} Bound  $\frac{1}{N^d} \|a_N^\infty-\langle a_N\rangle\operatorname{Id}_{\Gamma_N}\|^2_{\rm HS}$ in terms of $\|a_N\|_\infty$. 
			\end{flushleft}
\vspace{0.1cm} 

\begin{flushleft}
	\textbf{4.1.} Select a sine basis $\Si_N$ to  decompose $\A_{\Gamma_N}$ as $\Si_N\Lambda_N\Si_N^*$.
\end{flushleft}
  Recall the one-dimensional adjacency matrix  $\Delta_N$,
  $$	\Delta_N(x,y):=\begin{cases}
      1,  \ \ & |x-y|=1, \  x,y\in [[1,N]],\\
      0, \ \  &\text{otherwise}.
  \end{cases} $$
	A direct computation shows that $\Delta_N$ has $N$ simple eigenvalues $\lambda_N^{(k)}=2\cos\frac{k\pi}{N+1}$ corresponding to orthonormal eigenfunctions 
\begin{equation}\label{1716}
    s_N^{(k)}(x)=\sqrt{\frac{2}{N+1}}\sin\frac{k\pi x}{N+1} \textrm{ for  } 1\leq k\leq N.
\end{equation}
	For $\bk\in \Gamma_N$, define 
	\begin{equation}\label{sb}
		s_N^{(\bk)}(\x)=\mathop{\otimes}\limits_{l=1}^d s_N^{(k_l)}(x_l):= \left(\frac{2}{N+1}\right)^{d/2}\prod_{l=1}^d\sin\frac{k_l\pi x_l}{N+1}.
	\end{equation}
Since $\A_{\Gamma_N}=\sum_{l=1}^d\Delta_{N,l}$, where 
  $$	\Delta_{N,l}(\x,\y):=\begin{cases}
      1, \ \ & \|\x-\y\|_1=|x_l-y_l|=1, \  \x,\y\in \Gamma_N,\\
      0, \ \ &\text{otherwise}.
  \end{cases} $$
It follows that  $s_N^{(\bk)}$ ($\bk\in \Gamma_N$) is an  orthonormal basis of $\ell^2(\Gamma_N)$ satisfying 
	\begin{equation}\label{tez}
		\A_{\Gamma_N}s_N^{(\bk)}=\sum_{l=1}^d\Delta_{N,l}\mathop{\otimes}\limits_{l=1}^d s_N^{(k_l)}=\lambda_N^{(\bk)}s_N^{(\bk)} \textrm{ with } \lambda_N^{(\bk)}=\sum_{l=1}^d\lambda_N^{(k_l)}.
	\end{equation}
	Define the $(\Gamma_N\times\Gamma_N)$-matrices $\Si_N$ and  $\Lambda_N$
	$$\Si_N(\m,\bk):=s_N^{(\bk)}(\m), \ \ \Lambda_N:=\operatorname{diag}(\lambda_N^{(\bk)}). $$
	Then $\Si_N$ is unitary and  it follows from \eqref{tez} that 
	$$\A_{\Gamma_N}=\Si_N\Lambda_N\Si_N^*.$$
\begin{flushleft}
	\textbf{4.2.} Obtain the  explicit expression of $a_N^\infty$ in terms of the discrete  Fourier coefficients of $a_N$ by this decomposition.
\end{flushleft}
	
	Since  $\A_{\Gamma_N}=\Si_N\Lambda_N\Si_N^*$, we have  
	\begin{equation}\label{1716}
		e^{-{\rm i}t\A_{\Gamma_N}}a_Ne^{{\rm i}t\A_{\Gamma_N}}=\Si_Ne^{-{\rm i}t\Lambda_N}\Si_N^*a_N\Si_Ne^{{\rm i}t\Lambda_N}\Si_N^*.
	\end{equation}
Denote  the center term  $\Ce_N:=\Si_N^*a_N\Si_N$. By definition, 
	\begin{align}
		\nonumber	\Ce_N(\bk,\m)
		\nonumber=	& \sum_{\x\in\Gamma_N}s_N^{(\bk)}(\x)a_N(\x,\x)s_N^{(\m)}(\x)\\
		=&\left(\frac{2}{N+1}\right)^{d} \sum_{\x\in\Gamma_N} a_N(\x,\x)\prod_{l=1}^d\frac{\left( e^{\frac{k_l\pi x_li}{N+1}}-e^{-\frac{k_l\pi x_li}{N+1}}\right)\left( e^{\frac{m_l\pi x_li}{N+1}}-e^{-\frac{m_l\pi x_li}{N+1}}\right)}{(2i)^2}.\label{1444}
	\end{align}
	We introduce the following notation:
	For $\BT,\BV\in \R^d$, $\x\in \Gamma_N$, denote 
	\begin{align*}
		\BT\odot\BV&:=(\theta_l\varepsilon_l)_{l=1}^d,\\ 
		e^{(\BT)}(\x)&:=e^{\langle\BT,\x\rangle\pi {\rm i}},\\
		e^{(\BT)}\bullet a_N&:=\sum_{\x \in \Gamma_N}	e^{(-\BT)}(\x)a_N(\x,\x).
	\end{align*}
	Then  \eqref{1444} can be written as 
	\begin{equation}\label{1912}
		\left(\frac{1}{2(N+1)}\right)^{d}\sum_{\BV,\BV'\in \{1,-1\}^d} \sign(-\BV\odot\BV')	 e^{(\frac{\bk\odot\BV}{N+1}+\frac{\m\odot\BV'}{N+1})}\bullet a_N,
	\end{equation}
	where $\sign(\BV)=\prod_{l=1}^d\varepsilon_l$ for $\BV\in \{1,-1\}^d$. 
	It follows from \eqref{1716} that 
	\begin{align}
		\nonumber	a_N^\infty&=\lim_{T\to \infty} \frac{1}{T} \int_0^T e^{-{\rm i}t\A_{\Gamma_N}}a_Ne^{{\rm i}t\A_{\Gamma_N}}dt\\
		\nonumber	&=\Si_N \left(\lim_{T\to \infty} \frac{1}{T} \int_0^T e^{-{\rm i}t\Lambda_N}\Ce_Ne^{{\rm i}t\Lambda_N}dt\right) \Si_N^*\\
		\nonumber	&=\Si_N \left(\lim_{T\to \infty} \frac{1}{T} \int_0^T \Ce_N(\bk,\m) e^{{\rm i}t( \lambda_N^{(\m)}-  \lambda_N^{(\bk)} )}dt\right) \Si_N^*\\
		&=\Si_N \left(\Ce_N(\bk,\m) \delta_{\lambda_N^{(\bk)},\lambda_N^{(\m)}} \right) \Si_N^*, \label{1412}
	\end{align}
	where we use $$\lim_{T\to \infty} \frac{1}{T} \int_0^T e^{{\rm i}t\lambda}dt=\begin{cases}
		1, \ \ \lambda=0, \\
		0, \ \ \lambda\in \R\setminus\{0\}, 
	\end{cases}$$
	in the last line of the equation.
    \vspace{2mm}
	
		 Observe from \eqref{1912} that each  element of $\Ce_N$ is a linear combination of the discrete Fourier coefficients of $a_N$:   $e^{(\BT)}\bullet a_N$ with frequency  $\BT\in \frac{1}{N+1} [[-2N,2N]]^d$.
		 Thus 
	we can decompose \begin{equation}\label{12}
		\Ce_N(\bk,\m) \delta_{\lambda_N^{(\bk)},\lambda_N^{(\m)}}=\sum_{\BT\in \frac{1}{N+1} [[-2N,2N]]^d}\Ce_{N,\BT}
	\end{equation}
	with 
	\begin{equation}\label{fj}
		\Ce_{N,\BT}(\bk,\m)= 
		\delta_{\lambda_N^{(\bk)},\lambda_N^{(\m)}} \sum_{\substack{\BV,\BV'\in \{1,-1\}^d:\\
				\frac{\bk\odot\BV}{N+1}+\frac{\m\odot\BV'}{N+1}=\BT}}\frac{\sign(-\BV\odot\BV')e^{(\BT)}\bullet a_N}{\left(2(N+1)\right)^{d}}. 
	\end{equation}

\vspace{0.1cm} 
\begin{flushleft}
	\textbf{4.3.} Bound $\Ce_{N,\BT}(\bk,\m)$ through eigenvalue  analysis. 
\end{flushleft}
	\begin{itemize}[leftmargin=2em]
		\item Case $\BT={\bm 0}$: By \eqref{fj}, for $\Ce_{N,\bm 0}(\bk,\m) $  not vanishing, we have  $$\frac{\bk\odot\BV}{N+1}+\frac{\m\odot\BV'}{N+1}={\bm 0}$$ for some   $\BV,\BV'\in \{1,-1\}^d$.
	
		Since $\bk,\m\in \Gamma_N$, the above equation yields 
		$$\bk=\m, \ \ \BV=-\BV' .$$
	For any fixed $\bk= \m$,  there are $2^d$ pairs of  $\BV,\BV' \in \{1,-1\}^d$ such that  $\BV=-\BV'$, for  which  cases $\sign(-\BV\odot\BV')=1$.  	Thus \begin{align}\label{1214}
			\Ce_{N,\bm 0}=\operatorname{diag}\left(\frac{2^de^{({\bm 0})} \bullet a_N}{\left(2(N+1)\right)^{d}}\right)=\left(\frac{N}{N+1}\right)^d\langle a_N\rangle \operatorname{Id}_{\Gamma_N}
		\end{align}
		\item Case $ \BT\in \frac{1}{N+1} [[-2N,2N]]^d\setminus \{{\bm 0}\}$: By \eqref{fj}, for $\Ce_{N,\BT}(\bk,\m) $  not vanishing, we have  
		\begin{align}
			\lambda_N^{(\bk)}&=\lambda_N^{(\m)},\label{te} \\  \frac{\bk\odot\BV}{N+1}+\frac{\m\odot\BV'}{N+1}&=\BT\label{xia}
		\end{align}
		for some   $\BV,\BV'\in \{1,-1\}^d$.  
		
	We need  the following lemma on the  upper bound for the number of  $(\bk,\m)$ satisfying \eqref{te} and \eqref{xia}.
		\begin{lem}\label{c1}
			For any fixed    $\BV,\BV'\in \{1,-1\}^d$, there are at most $2N^{d-1}$  pairs of  $(\bk,\m)$ satisfying \eqref{te} and \eqref{xia}.
		\end{lem}
		\begin{proof}[Proof of Lemma \ref{c1}]
			Fix a pair of $\BV,\BV'$.  Then through \eqref{xia}, $\bk$ can be determined by $\m$:   \begin{equation}\label{fenl}
				\frac{k_l}{N+1}=-\frac{\varepsilon_{l}\varepsilon_{l}'m_l}{N+1}+\varepsilon_{l}\theta_l, \ \ 1\leq l\leq d.
			\end{equation}
			Without loss of generality, 	suppose $\theta_1\neq 0$. Then for any fixed $(m_2,\cdots, m_d)$ ($N^{d-1}$  choices), it follows from \eqref{te} that  
			$$\cos\frac{k_1\pi}{N+1}-\cos\frac{m_1\pi}{N+1}=\sum_{l=2 }^d\left(\cos\frac{m_l\pi}{N+1}-\cos\frac{k_l\pi}{N+1}\right).$$
			Substituting \eqref{fenl} to the above equation yields 
			\begin{equation}\label{jix}
				\cos\left(\frac{m_1\pi}{N+1}-\varepsilon_{1}'\theta_1\pi\right)-\cos\frac{m_1\pi}{N+1}=\star,
			\end{equation}
			where $\star$ is some fixed number after $\BT,\BV,\BV',(m_2,\cdots, m_d)$ fixed. It follows from \eqref{jix} that 
			$$2	\sin\left(\frac{m_1\pi}{N+1}-\frac{\varepsilon_{1}'\theta_1\pi}{2}\right)\cdot\sin\frac{\varepsilon_{1}'\theta_1\pi}{2}=\star.$$
			Since $0\neq |\theta_1|<2$, we have $\sin\frac{\varepsilon_{1}'\theta_1\pi}{2}\neq 0$. Thus there are at most two solutions of $m_1\in [[1,N]]$ for the above equation.
		\end{proof}
		Since there are at most $4^d$  choices  of  $\BV,\BV'\in \{1,-1\}^d$, it follows from  \eqref{fj} and Lemma \ref{c1} that the matrix  $\Ce_{N,\BT}$ has nonzero elements on at most $2\cdot4^dN^{d-1}$ sites,  and each element $\Ce_{N,\BT}(\bk,\m)$ is  bounded by $\left(\frac{2}{N+1}\right)^d|e^{(\BT)}\bullet a_N|$. 
	\end{itemize}
    
	By \eqref{1912},  each  element of  $\Ce_N$ is a combination of $e^{(\BT)}\bullet a_N$ for at most $4^d$ many $\BT$'s. Recall \eqref{1412}, \eqref{12}, \eqref{1214} and the above estimate of  $\Ce_{N,\BT}$ ($\BT\neq {\bm 0}$).  Applying Cauchy inequality to each element yields 
	\begin{align}\nonumber&\|a_N^\infty-\langle a_N\rangle\operatorname{Id}_{\Gamma_N}\|^2_{\rm HS}\\
		=&\|\Ce_N(\bk,\m)\nonumber \delta_{\lambda_N^{(\bk)},\lambda_N^{(\m)}}-\langle a_N\rangle\operatorname{Id}_{\Gamma_N}\|^2_{\rm HS}\\
		\leq &2\|\sum_{\BT\neq {\bm 0}}\Ce_{N,\BT}\|^2_{\rm HS}+2\|\Ce_{N,{\bm 0}}-\langle a_N\rangle\operatorname{Id}_{\Gamma_N}\|^2_{\rm HS}\nonumber\\
		\leq& 2\cdot 4^d\sum_{\BT\neq {\bm 0}}\|\Ce_{N,\BT}\|^2_{\rm HS}+2\left(1-\left(\frac{N}{N+1}\right)^d\right)^2|\langle a_N\rangle|^2 N^d\nonumber\\
		\leq & \left(2\cdot 4^d\right)^2\cdot N^{d-1}\left(\frac{2}{N+1}\right)^{2d}\sum_{\BT\neq {\bm 0}}|e^{(\BT)}\bullet a_N|^2+ d^2N^{d-2}\|a_N\|_\infty^2. \label{51}
	\end{align}
	Thus it remains to estimate the summation \begin{equation*}
		\left(\frac{2}{N+1}\right)^{2d}\sum_{\BT\in \frac{1}{N+1} [[-2N,2N]]^d}|e^{(\BT)}\bullet a_N|^2.
	\end{equation*}
	Define  $\widetilde{e^{(\BT)}},\widetilde{a_N}\in \ell^2([[0,N]]^{d})$ as 
	\begin{align*}
		\widetilde{e^{(\BT)}}(\x)&:=\left(\frac{1}{N+1}\right)^{\frac{d}{2}}e^{\langle\BT,\x\rangle\pi {\rm i}},\ \quad \ \  \x\in  [[0,N]]^{d}\\
		\widetilde{a_N}(\x)&:=\begin{cases}
			\left(\frac{1}{N+1}\right)^{\frac{d}{2}}a_N(\x,\x), \ \ &\x\in\Gamma_N, \\
			0, \ \ &\x\in [[0,N]]^{d}\setminus\Gamma_N.
		\end{cases}
	\end{align*}
	Then \begin{equation}\label{1224}
		\left(\frac{1}{N+1}\right)^{2d}|e^{(\BT)}\bullet a_N|^2=\left|\langle\widetilde{e^{(\BT)}},\widetilde{a_N}\rangle\right|^2.
	\end{equation}
\vspace{0.1cm} 

\begin{flushleft}
		\textbf{4.4.}  Partition  Fourier coefficients into orthogonal classes, and obtain bounds in each class. 
\end{flushleft}
	 
	Although $\widetilde{e^{(\BT)}}$  are not mutually orthogonal,  we demonstrate that they can be classified into $4^d$ classes  within each of which orthogonality is preserved.
	First, classify the set  $\frac{1}{N+1} [[-2N,2N]]^d$ into $2^d$ subsets of the form $\prod_{l=1}^d I_l$, where each $I_l$ is either  $\frac{1}{N+1} [[-2N,0]]$  or $\frac{1}{N+1} [[0,2N]]$. For any two $\BT,\BT'$ within the  same subset,  the coordinate-wise difference satisfies $|\theta'_l-\theta_l|<2$ for any $l$. Next,  define the following   equivalence relation on each subset:
	$$\BT\sim\BT'\Longleftrightarrow 2\mid(N+1)(\theta'_l-\theta_l),\ \ \forall l. $$ This equivalence relation further partitions  each subset into $2^d$ equivalence classes.
	Orthogonality within each equivalence class follows from the inner product expression:
	$$	\langle\widetilde{e^{(\BT)}},\widetilde{e^{(\BT')}}\rangle=\left(\frac{1}{N+1}\right)^{d}\prod_{l=1}^d\left(\sum_{x_l=0}^Ne^{(\theta'_l-\theta_l)x_l\pi {\rm i}}\right).$$  
	When $\BT\sim \BT'$, $\BT\neq\BT'$ (say, $\theta_l\neq\theta'_l$),  the summation in the  $l$-coordinate simplifies to zero due to the imposed equivalence condition, ensuring orthogonality.
	Thus by Bessel inequality, we have 
	\begin{equation}\label{1227}
		\sum_{\BT\in \frac{1}{N+1} [[-2N,2N]]^d}\left |\langle\widetilde{e^{(\BT)}},\widetilde{a_N}\rangle\right |^2\leq 4^d\| \widetilde{a_N}\|_{\ell^2}^2\leq 4^d\|a_N\|_\infty^2.
	\end{equation}
	It  follows from \eqref{51}, \eqref{1224} and \eqref{1227} that 
	\begin{equation*}
		\limsup_{N \rightarrow \infty} \frac{1}{N^d} \|a_N^\infty-\langle a_N\rangle\operatorname{Id}_{\Gamma_N}\|^2_{\rm HS}\leq 	\lim_{N \rightarrow \infty} \frac{1}{N}C_d\|a_N\|_\infty^2=0, 
	\end{equation*}
This 	verifies \eqref{=}. Thus we finish the proof of Theorem \ref{thm}.

	\section{ Eigenfunction correspondence}\label{pf2}
In this section, we establish the eigenfunction correspondence to relate Dirichlet  conditions problems to known results in the periodic conditions  case, thereby proving 
Theorems \ref{thm.}, and \ref{ps}. In particular, this provides an alternative proof of Theorem  \ref{thm}.

 \subsection{Established results in the periodic conditions case}\label{est} Let us first state the established   quantum ergodicity results    for $\mathfrak{A}_{\Gamma_N}$, the adjacency matrix with periodic condition, which is defined as:
\begin{defn}\label{def}
	For $\psi\in \ell^2(\Gamma_{N})$, denote by ${\psi}^{\rm P}$ its  periodic extension  concerning $\Gamma_N$ to $\Z^d$ such that
\begin{align*}
	{\psi}^{\rm P}(\x)	:&={\psi}(\x), \ \ \x \in \Gamma_N	\\
	{\psi}^{\rm P}(\x+N\bm e_l):&={\psi}^{\rm P}(\x), \ \ \x\in\Z^d, \  1\leq l\leq d.
\end{align*} 

Define  $\mathfrak{A}_{\Gamma_N}$ as    
\begin{equation}\label{pero}
	(\mathfrak{A}_{\Gamma_N}\psi)(\x):=\sum_{\substack{\y\in \Z^d \\ \|\y-\x\|_1=1 }}{\psi}^{\rm P}(\y), \ \ \psi\in \ell^2(\Gamma_N),\   \x\in \Gamma_N.  
\end{equation}
\end{defn}
The following  quantum ergodicity results for $\mathfrak{A}_{\Gamma_{N}}$ was proved in \cite{MS23}:
	\begin{thm}[\cite{MS23}, diagonal observables]\label{tmm}
	 The statement of Theorem \ref{thm} holds 	with  $\A_{\Gamma_N}$ replaced by $\mathfrak{A}_{\Gamma_N}$.
	\end{thm}
\begin{thm}[\cite{MS23}, finite range  observables]\label{nond}
	Recall  the definition of $L_{\z}$ (c.f.,  \eqref{Lz}).  The statement of Theorem \ref{thm.} holds with   $\A_{\Gamma_N}$ replaced by $\mathfrak{A}_{\Gamma_N}$ and the correlation \eqref{dco} replaced by 
	\begin{equation}\label{pco} 
		\langle K_N\rangle_{\psi_N^{(j)}}^{{\rm P}}:= \frac{1}{\#\Gamma_{N}}\sum_{\z } \sum_{\x \in L_{\z} }  K_N(\x,\x+\z) \langle\psi_N^{(j)},\tau_{\z}\psi_N^{(j)}\rangle,
	\end{equation} where  the periodic translation $\tau_{\z} : \ell^2(\Gamma_N)\to\ell^2(\Gamma_N) $ is defined  as 
	$$(	\tau_{\z}\psi)(\x):={\psi}^{\rm P}(\x+\z),\ \  \psi\in \ell^2(\Gamma_N), \ \x\in \Gamma_{N}. $$

\end{thm}
\begin{rem}
	One should note the small difference between $\rho_{\z}$ in  \eqref{dco} and  $\tau_{\z}$ in \eqref{pco}. In fact, combining  $\|\tau_{\z}-\rho_{\z}\|_{\rm HS}^2=O(N^{d-1}) $ and \eqref{4}, we can  prove that if \eqref{1.} holds for the correlation given by \eqref{dco}, then it also holds for the correlation given by \eqref{pco}, 	and vice versa. Here we present  Theorem \ref{nond}  following the original statement  of \cite{MS23}.
\end{rem}
\subsection{The injective map $\he$ and eigenfunction correspondence}\label{map}
We establish an injective map  $\he$ from $\mathcal{A}_{\Gamma_N}$-eigenfunctions to $\mathfrak{A}_{\Gamma_{2N+2}}$-eigenfunctions through zero-extension and symmetric reflection of eigenfunctions.
	\begin{defn}\label{df}
		Denote by $\mathfrak{R}_{N}^l:\Gamma_{2N+2}\to \Gamma_{2N+2}$ the reflection on the $l$-coordinate:
		$$\mathfrak{R}_{N}^l(\x)_{l'}:=\begin{cases}
			x_{l'}, \ \ &l'\neq l,\\
			2N+2-x_{l'}, \ \ &l'= l . 
		\end{cases}$$
		Define the injective map  $$\he:\ell^2(\Gamma_N)\to \ell^2(\Gamma_{2N+2})$$ 
		such that for $\psi_N\in \ell^2(\Gamma_N)$, $\he\psi_N\in \ell^2(\Gamma_{2N+2})$ satisfies  \eqref{5}--\eqref{7}, 
		\begin{align}
			\label{5}	(\he\psi_N)(\x)&=2^{-\frac{d}{2}}\psi_N(\x), &   \x \in \Gamma_N,\\ 
			\label{6}		(\he\psi_N)(\x)&=0,  &\exists l,\  N+1\mid x_l, \\ 
			(\he\psi_N)(\mathfrak{R}_{N}^l(\x))&=-(\he\psi_N)(\x), &\forall 1\leq l\leq d. \label{7}
		\end{align}
		Such  $\he\psi_N$ is existent and  unique.
	\end{defn} 
	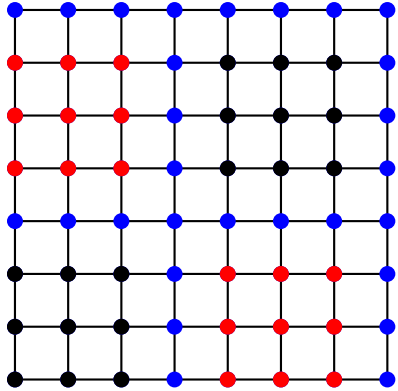
\begin{figure}[htp]
		\centering
		\begin{tikzpicture}[scale=0.7]

			\draw[step=1, black, thick] (0,0) grid (7,7);

			\foreach \x in {0,1,...,7} {
				\foreach \y in {0,1,...,7} {
			
					\ifnum \x>-1  \ifnum \y>-1 
					\fill[blue] (\x,\y) circle (0.15);
					\fi\fi
				
					\ifnum \x<3 \ifnum \y<3
					\fill[black] (\x,\y) circle (0.15);
					\fi\fi
					\ifnum \x>3 	\ifnum 7>\x  \ifnum \y>3 \ifnum 7>\y
					\fill[black] (\x,\y) circle (0.15);
					\fi\fi\fi\fi

					\ifnum \x<3 \ifnum \y<7  \ifnum \y>3
					\fill[red] (\x,\y) circle (0.15);
					\fi\fi\fi
					\ifnum \y<3 \ifnum \x<7  \ifnum \x>3
					\fill[red] (\x,\y) circle (0.15);
					\fi\fi\fi

				}
			}
			
		\end{tikzpicture}
		\caption{A description of the map $\he$ for $d=2,N=3$:   	$\he\psi$ takes $0$ at the blue points,
			the same sign as $\psi$ at the black points, and the positive sign to $\psi$ at the red points.
		}
	\end{figure}
	By the definition of $\he$, we have the following eigenfunction correspondence:
	\begin{prop}\label{eb}
		Let $\psi_N$ be a normalized eigenfunction of $\A_{\Gamma_N}$ corresponding to the eigenvalue $\lambda$.  Then 	$\he\psi_N$ is a normalized eigenfunction of $\mathfrak{A}_{\Gamma_{2N+2}}$ corresponding to the eigenvalue $\lambda$. Moreover, if $\{\psi_N^{(j)}\}_{j=1}^{N^d}$ is an orthonormal basis of $\ell^2(\Gamma_N)$ consisting of eigenfunctions of $\A_{\Gamma_N}$, then $\{\he\psi_N^{(j)}\}_{j=1}^{N^d}$ is an orthonormal class of $\ell^2(\Gamma_{2N+2})$ consisting of eigenfunctions of $\mathfrak{A}_{\Gamma_{2N+2}}$. 
	\end{prop}

\subsection{Proof of Theorem \ref{thm}}\label{proof1}
With Theorem  \ref{tmm} and Proposition \ref{eb}, we are ready to prove Theorem \ref{thm}.

	\begin{proof}[Proof of Theorem \ref{thm}]
		Consider the diagonal observable $\mathfrak{a}_N$ on $\Gamma_{2N+2}$ defined by 
		$$\mathfrak{a}_N(\x,\x):=\begin{cases}
			a_N(\x,\x), \ \ &\x\in \Gamma_N,\\
			0, \ \ &\x\in \Gamma_{2N+2}\setminus \Gamma_{N} . 
		\end{cases}$$
		Let $\{\psi_N^{(j)}\}_{j=1}^{N^d}$ be an orthonormal basis of $\ell^2(\Gamma_N)$ consisting of eigenfunctions of $\A_{\Gamma_N}$. Since  $\{\he\psi_N^{(j)}\}_{j=1}^{N^d}$ is an orthonormal class of $\ell^2(\Gamma_{2N+2})$ consisting of eigenfunctions of $\mathfrak{A}_{\Gamma_{2N+2}}$ by Proposition \ref{eb}, we can extend $\{\he\psi_N^{(j)}\}_{j=1}^{N^d}$ to an orthonormal basis of $\ell^2(\Gamma_{2N+2})$ consisting of eigenfunctions of $\mathfrak{A}_{\Gamma_{2N+2}}$. Denote the basis by $\{\varphi_N^{(j)}\}_{j=1}^{(2N+2)^d}$, where $\varphi_N^{(j)}=\he\psi_N^{(j)}$ for $1\leq j\leq N^d$.    By definition, we have 
		\begin{align*}
			\langle \psi_N^{(j)}, a_N \psi_N^{(j)}\rangle&=2^d\langle \he\psi_N^{(j)}, \mathfrak{a}_N \he\psi_N^{(j)}\rangle, \quad 
			\langle a_N\rangle=\frac{(2N+2)^d}{N^d}\langle \mathfrak{a}_N\rangle.
		\end{align*}
		Thus,  \begin{align*}
			& \frac{1}{N^d} \sum_{j=1}^{N^d}\left| \langle \psi_N^{(j)}, a_N \psi_N^{(j)}\rangle-\langle a_N\rangle\right|^2\\
			=&\frac{1}{N^d} \sum_{j=1}^{N^d}4^d \left| \langle \he\psi_N^{(j)}, \mathfrak{a}_N \he\psi_N^{(j)}\rangle-\frac{(N+1)^d}{N^d}\langle \mathfrak{a}_N\rangle\right|^2\\
			\leq &\frac{1}{N^d} \sum_{j=1}^{N^d}4^d \left(2\left| \langle \he\psi_N^{(j)}, \mathfrak{a}_N \he\psi_N^{(j)}\rangle-\langle \mathfrak{a}_N\rangle\right|^2+2\left|\left (\frac{(N+1)^d}{N^d}-1\right )\langle \mathfrak{a}_N\rangle\right|^2\right)\\
			\leq &   \frac{16^d}{(2N+2)^d}  \sum_{j=1}^{(2N+2)^d} \left| \langle \varphi_N^{(j)}, \mathfrak{a}_N \varphi_N^{(j)}\rangle-\langle \mathfrak{a}_N\rangle\right|^2+O\left(\frac{1}{N^2}\right).
		\end{align*}
		By Theorem \ref{tmm}, the first term in  the last line of the inequality tends to $0$ as $N\to \infty$, which implies Theorem \ref{thm}.
	\end{proof} 

\subsection{Proof of Theorem \ref{thm.}}\label{pf21}
In this part, we extend the eigenfunction correspondence to establish  quantum ergodicity for  finite range observables, thereby proving Theorem \ref{thm.}.

\begin{proof}[Proof of Theorem \ref{thm.}]
We consider the finite range observable $\mathcal{K}_N$  on $\Gamma_{2N+2}$ defined by 
 $$\mathcal{K}_N(\x,\y):=\begin{cases}
 	K_N(\x,\y), \ \ &\x,\y\in \Gamma_N,\\
 	0, \ \ &\text{otherwise} , 
 \end{cases}$$
Recalling  Definition \ref{df},  we have 
 \begin{align*}
 	\langle \psi_N^{(j)}, K_N \psi_N^{(j)}\rangle&=2^d\langle \he\psi_N^{(j)}, \mathcal{K}_N \he\psi_N^{(j)}\rangle.
 \end{align*}
	Thus, it follows from Theorem \ref{nond} and Proposition \ref{eb} that   \begin{align}\label{l3}
	\lim _{N \rightarrow \infty} \frac{1}{N^d} \sum_{j=1}^{N^d}\left| \langle \psi_N^{(j)}, K_N \psi_N^{(j)}\rangle-2^d \langle\mathcal{ K}_N\rangle_{\he \psi_N^{(j)}}^{\rm P}\right|^2=0.
\end{align}
Next,  we introduce the  averaged observable $\widetilde{K_N}$ of $K_N$ for $\|\z\|_1\leq R$: 
\begin{equation*}
	\widetilde{K_N}(\x,\x+\z):=
\frac{1}{\#L_{\z}} \sum_{\y\in L_{\z}}K_N(\y,\y+\z),  \ \ \x\in L_{\z},
\end{equation*}
 which is also equivalent to (recalling  \eqref{Lz}, the definition of $\rho_{\z}$)
\begin{equation}\label{dingyi}
	\widetilde{K_N}=\sum_{\z}  \sum_{\x\in L_{\z}}  \frac{1}{\#L_{\z}}K_N(\x,\x+\z) \rho_{\z}.
\end{equation}
 That is, $\widetilde{K_N}$ takes the average of $K_N$ on  $\{(\x,\x+\z):\  \x\in L_{\z}\}$. Particularly, we have   
\begin{equation}\label{xd}
	\sum_{\x\in L_{\z}}K_N(\x,\x+\z)=\sum_{\x\in L_{\z}}\widetilde{K_N}(\x,\x+\z)
\end{equation}
Consider   $\widetilde{\mathcal{K}_N}$ on  $\Gamma_{2N+2}$ defined by    $$\widetilde{\mathcal{K}_N}(\x,\y):=\begin{cases}
 	\widetilde{K_N}(\x,\y), \ \ &\x,\y\in \Gamma_N,\\
 	0, \ \ &\text{otherwise} .
 \end{cases}$$
Analogous to \eqref{l3}, we have  \begin{align}\label{l2}
 	\lim _{N \rightarrow \infty} \frac{1}{N^d} \sum_{j=1}^{N^d}\left| \langle \psi_N^{(j)}, \widetilde{K_N} \psi_N^{(j)}\rangle-2^d \langle\widetilde{\mathcal{ K}_N}\rangle_{\he \psi_N^{(j)}}^{\rm P}\right|^2=0.
 \end{align}
Recall the definition of $ \langle K\rangle_{ \psi}^{\rm P} $ (c.f., \eqref{pco}). It follows from  \eqref{xd} that  
\begin{equation}\label{l1}
	 \langle\mathcal{ K}_N\rangle_{\he \psi_N^{(j)}}^{\rm P}=\langle\widetilde{\mathcal{ K}_N}\rangle_{\he \psi_N^{(j)}}^{\rm P}.
\end{equation}
By \eqref{l3}, \eqref{l2} and \eqref{l1}, we have 
\begin{align}\label{linmen}
	\lim _{N \rightarrow \infty} \frac{1}{N^d} \sum_{j=1}^{N^d}\left| \langle \psi_N^{(j)}, K_N \psi_N^{(j)}\rangle-\langle \psi_N^{(j)}, \widetilde{K_N} \psi_N^{(j)}\rangle\right|^2=0.
\end{align}
Substituting \eqref{dingyi} into \eqref{linmen}, we finish the proof of Theorem \ref{thm.}.
\end{proof}

	\subsection{Proof of Proposition \ref{app} and Theorem \ref{ps}}\label{pf22}
In this part, we discuss the partial quantum ergodicity for periodic Schr\"odinger operators.

	First, we prove Proposition \ref{app} by giving   an example that   periodic Schr\"odinger operator $\mathcal{H}_{\Lambda_N}$ has  no almost equidistributed eigenfunction. This shows that the ``partial'' is necessary.
\begin{proof}[Proof of Proposition \ref{app}]
		Let $d=1$ and $q_1=2$. Consider the $2$-periodic potential 
	$$V(\x)= \begin{cases}
		M, \ \ &2\mid\x,\\
		0, \ \ &2\nmid\x,	 \end{cases}$$
	for some  $M\gg1$. 
	By perturbation theory for self-adjoint operators, it follows from $\|\A_{\Lambda_{N}}\|\leq 2$ that $\mathcal{H}_{\Lambda_N}$ has $N$ eigenvalues in $[M-2,M+2]$ and $N$ eigenvalues in $[-2,2].$
	Assume that $\psi$ is a $\ell^2$-normalized eigenfunction of $\mathcal{H}_{\Lambda_N}$ corresponding to $\lambda\in[-2,2]$. For $2\mid\x$, we have 
	$$(M-\lambda)\psi(\x)=-\sum_{\substack{\y\in \Lambda_N \\ |\y-\x|=1 }}\psi(\y).$$
	Thus by  Cauchy inequality, $$|M-\lambda|^2|\psi(\x)|^2\leq2 \sum_{\substack{\y\in \Lambda_N \\ |\y-\x|=1 }}|\psi(\y)|^2.$$
	Summing the above equality for $\x$ such that  $2\mid\x$ yields 
	$$\sum_{2\mid\x }|\psi(\x)|^2\leq \frac{4}{|M-\lambda|^2}\|\psi\|_{\ell^2}^2\leq \frac{4}{|M-2|^2}.$$
	Thus, $$\sum_{2\nmid\x }|\psi(\x)|^2\geq1- \frac{4}{|M-2|^2}.$$
	By the same argument, for any  eigenfunction ${\psi'}$ corresponding to eigenvalue  ${\lambda'}\in [M-2,M+2]$, we have 
	$$\sum_{2\mid\x }|{\psi'}(\x)|^2\geq1- \frac{4}{|M-2|^2}.$$
	Since $M\gg1$,  $\mathcal{H}_{\Lambda_N}$ has  $N$ eigenfunctions with  $\ell^2$-mass  concentrated on odd points and the $N$ eigenfunctions with  $\ell^2$-mass  concentrated on even points. Thus quantum ergodicity does not hold for $\mathcal{H}_{\Lambda_N}$.  
\end{proof}

Next, we use the eigenfunction correspondence  to prove Theorem \ref{ps}.
\begin{defn}
Recalling \eqref{pz}, let   $\mathfrak{H}_{\Lambda_N}$ be  the finite truncation of $H$ to $\Lambda_N$ with periodic condition 
$$(\mathfrak{H}_{\Lambda_N} \psi)(\x):=(\mathfrak{A}_{\Lambda_N}\psi)(\x)+V(\x) \psi(\x), \ \ \psi\in \ell^2(\Lambda_N),\   \x\in \Lambda_N,$$
where $\mathfrak{A}_{\Lambda_N}$ is defined analogously to  \eqref{pero} (the periodic extension ${\psi}^{\rm P}$  is now   concerning $\Lambda_N$ instead of $\Gamma_{N}$).
\end{defn}
The following result for  partial  quantum ergodicity of  $\mathfrak{H}_{\Lambda_{N}}$ was proved  by    \cite{MS23,Liu24}:
\begin{thm}[\cite{MS23,Liu24}]\label{tmm1}
The  statement of Theorem \ref{ps} holds  	with  $\mathcal{H}_{\Lambda_N}$ replaced by $\mathfrak{H}_{\Lambda_N}$ (assumption \eqref{cd} can be removed).
\end{thm}

Now we are ready to prove Theorem \ref{ps}.
	\begin{proof}[Proof of Theorem \ref{ps}]
	Select an appropriate block 
		$$\Omega_N:=\prod_{l=1}^d[[1,2q_lN+2]].$$
	We need to   redefine $ \mathfrak{R}_{N}^l $ and $\he$:

		Let  $\mathfrak{R}_{N}^l:\Omega_N\to \Omega_N$ be the reflection on the $l$-coordinate:
		$$\mathfrak{R}_{N}^l(\x)_{l'}:=\begin{cases}
			x_{l'}, \ \ &l'\neq l,\\
			2q_lN+2-x_{l'}, \ \ &l'= l . 
		\end{cases}$$
		Define the injective map $$\he:\ell^2(\Lambda_N)\to \ell^2(\Omega_N)$$ 
		such that for $\psi_N\in \ell^2(\Lambda_N)$, $\he\psi_N$ is the unique element in  $\ell^2(\Omega_N)$ satisfying \eqref{5`}--\eqref{7`}, 
		\begin{align}
			\label{5`}	(\he\psi_N)(\x)&=2^{-\frac{d}{2}}\psi_N(\x), &   \x \in \Lambda_N,\\ 
			\label{6`}		(\he\psi_N)(\x)&=0,  &\exists l,\  q_lN+1\mid x_l, \\ 
			(\he\psi_N)(\mathfrak{R}_{N}^l(\x))&=-(\he\psi_N)(\x), &\forall 1\leq l\leq d. \label{7`}
		\end{align}
By assumption  \eqref{cd} on the period lengths, $\{\he \psi_N^{(j)}\}_{j=1}^{\#\Lambda_N}$ is an orthonormal class of $\ell^2(\Omega_N)$ consisting of eigenfunctions of $\mathfrak{H}_{\Omega_N}$.
		
		 We note that $\mathfrak{H}_{\Omega_N}$ is not exactly the form in Theorem  \ref{tmm1} since $\mathfrak{H}_{\Omega_N}$ has $2N+2$ periods in the $1$-periodic directions and $2N+1$ periods in the $2$-periodic directions while the periodic operators considered in Theorem  \ref{tmm1}  have the same numbers of periods in all directions. This is not a matter since  Theorem  \ref{tmm1} remains true for $\mathfrak{H}_{\Omega_N}$ for  the following reasons:
		 
		   Following the proof of \cite{MS23}, one  should  modify  the criterion concerning Floquet eigenvalues in \cite[equation (1.3)]{MS23}  by 
		\begin{equation}\label{ass}
			\lim _{N \rightarrow \infty} \sup _{\substack{\bm\theta_1 \in \T^d_N \\ \bm\theta_1 \neq \bm 1}} \frac{\#\left\{\bm\theta_2 \in \T^d_N:\  E_s\left(  \bm\theta_1+\bm\theta_2    \right)=E_w\left(\bm\theta_2\right)\right\}}{N^d}=0,
		\end{equation} where 
		$$\T^d_N:= \left \{\bm\theta\in [0,1]^d :\ \theta_l\in \frac{q_l}{2q_lN+2} [[1 ,\frac{2q_lN+2}{q_l} ]]\right \} .$$
	Since   \eqref{ass} can be verified  by the same argument of \cite{Liu24}, partial quantum ergodicity holds for $\mathfrak{H}_{\Omega_N}$.
		
		 As  the  proof in Section \ref{proof1}, the remaining part of the  proof follows  from considering  the diagonal observable $\mathfrak{a}_N$ on $\Omega_N$ defined by 
		$$\mathfrak{a}_N(\x,\x):=\begin{cases}
			a_N(\x,\x), \ \ &\x\in \Lambda_N,\\
			0, \ \ &\x\in \Omega_N\setminus \Lambda_{N}.
		\end{cases}$$
	\end{proof}

\section{Further comments}\label{fur}

\subsection{About the eigenfunction correlators in Theorem \ref{thm.}}\label{FEC}

For large $(q+1)$-regular  graphs $G_N=(V_N,E_N)$ that are spectral expanders with tree-like structures,   \cite{Ana17}  proved quantum ergodicity of adjacency matrices $\A_{G_N}$ for finite range observables $K_N$:
\begin{equation}\label{tree}
		\lim _{N \rightarrow \infty} \frac{1}{\#V_N} \sum_{j=1}^{\#V_N}\left| \langle \psi_N^{(j)}, K_N \psi_N^{(j)}\rangle-\langle K_N\rangle_{\lambda_N^{(j)}}\right|^2=0,
\end{equation}
where the eigenfunction correlator $$\langle K_N\rangle_{\lambda_N^{(j)}}:=\frac{1}{\# V_N}\sum_{x,y\in V_N} K_N(x,y)\Phi_{\lambda_N^{(j)}}(d(x,y))$$ is a universal quantity that depends only on $K_N$ and the   eigenvalue  $\lambda_N^{(j)}$. Here $\Phi_\lambda(n)$ is the spherical function of the $(q+1)$-regular tree,    which has an explicit form in
terms of Chebyshev polynomials.  This result was further extended  to the non-regular graphs where the eigenfunction correlator is expressed in terms of Green's function on the universal covering tree \cite{AS19a}.

However, for the adjacency matrix on $\Z^d$ with periodic conditions, \eqref{tree} does not  hold. It was proved in \cite{MS23} that the eigenfunction correlator is basis-dependent and non-universal in the sense that it can not be expressed as a function of the corresponding eigenvalue $\lambda_N^{(j)}$.

    Notably, \eqref{tree} also fails for the sequence of  $2$-regular graphs (discrete  $N$-circles $C_N$)  since it is not a sequence of expanders (see also   the discussion in  the first paragraph of \cite[Section 6.1.1]{Ana22}). However,  we note   the difference   that \eqref{tree} holds for the Dirichlet-truncated adjacency matrix on $\Z$, that is, the graphs with one edge removed from $C_N$:
\begin{prop}\label{pwucha}
    
Let $d=1$.
Then all the  eigenvalues of $\A_{\Gamma_N}$ are simple and the eigenbasis $\{\psi_N^{(j)}\}_{j=1}^{N}$ is explicit: $$\psi_N^{(j)}(x)=s_N^{(j)}(x)=\sqrt{\frac{2}{N+1}}\sin\frac{j\pi x}{N+1}$$ corresponding to the eigenvalue $\lambda_N^{(j)}=2\cos\frac{j\pi}{N+1}$.

Moreover, for any $|z|\leq R$, we have the approximation:  \begin{equation}\label{wucha}
	\langle s_N^{(j)},\rho_{z}s_N^{(j)}\rangle=\cos\left(\frac{|z|j\pi}{   N+1}\right)+O\left(\frac{1}{N}\right).
\end{equation}
Since the corrector allows for perturbations within $o(1)_{N\to\infty}$, it follows from  \eqref{wucha} that  replacing $ \langle\psi_N^{(j)},\rho_{\z}\psi_N^{(j)}\rangle$ in \eqref{dco} with $\cos\frac{|z|j\pi}{N+1}$ provides an explicit expression for the correlator in the one-dimensional case. Note that $\cos\frac{|z|j\pi}{N+1}=\Phi_{\lambda_N^{(j)}}(|z|)$, where $\Phi_\lambda(n)=\cos\left(n\arccos\frac{\lambda}{2}\right)$ is the  spherical function of $\Z$. Thus \eqref{tree} holds for $\A_{\Gamma_N}$.
\end{prop}
\begin{proof}[Proof of   Proposition  \ref{pwucha}]
We will use  a property of the the spherical function (equation \eqref{rec} below). We refer to  \cite[Section 2.3]{AS19b} for a more general version on  the $(q+1)$-regular tree. Here we
include the proof of \eqref{rec} for the reader's convenience.  

  Recall that the spherical function  of $\Z$: $\Phi_\lambda(n)=\cos\left(n\arccos\frac{\lambda}{2}\right)$ satisfies $\Phi_\lambda(0)=1$, $\Phi_\lambda(1)=\frac{\lambda}{2}$ and the recursive formula for $n\geq 1$
\begin{equation}\label{r1}
	\Phi_\lambda(n+1)=\lambda\Phi_\lambda(n)-\Phi_\lambda(n-1).
\end{equation}
Define the operator $T_n$ on $\ell^2(\Z)$ such that $T_0=\operatorname{Id}_{\Z}$ and for $n\geq 1$, 
$$T_n(x,y):=\begin{cases}
	\frac{1}{2}, \ \ &|x-y|=n, \\
	0, \ \ &\text{otherwise}.
\end{cases}$$
Note that  the operator $T_n$ satisfies  the recursive formula for $n\geq 1$
\begin{equation}\label{r2}
T_{n+1}=\A_\Z T_n-T_{n-1}.
\end{equation}
Using  \eqref{r1} and \eqref{r2} and proving by induction on $n$, we have the equation 
\begin{equation}\label{rec}
T_n=	\Phi_{\A_\Z}(n),
\end{equation}
where 	$\Phi_{\A_\Z}(n)$ is defined by replacing $\lambda$ by $\A_\Z$ in the polynomial $\lambda\mapsto 	\Phi_\lambda(n)$. 

Note that $$\A_{\Z}^n(x,y)=\#\{
\text{paths of length $n$  from $x$ to $y$ in $\Z$}\},$$
$$\A_{\Gamma_N}^n(x,y)=\#\{\text{paths of length $n$  from $x$ to $y$ in $\Gamma_N$}\}.$$
Thus for $|z|\leq R$, we have 
\begin{equation}\label{ns}
	\Phi_{\A_{\Gamma_N}}(|z|)(x,y)=\Phi_{\A_\Z}(|z|)(x,y)\ \ \text{ for $x\in [[1+R,N-R]]$, $y\in \Gamma_{N}$}
\end{equation}
and both $\Phi_{\A_{\Gamma_N}}(|z|)$ and $\Phi_{\A_\Z}(|z|)$ are supported at distance $\leq R$ from the diagonal.

Recall \eqref{1716} that $s_N^{(j)}(x)=\sqrt{\frac{2}{N+1}}\sin\frac{j\pi x}{N+1}$ is an eigenfunction of $\A_{\Gamma_N}$ with corresponding eigenvalue $\lambda_N^{(j)}=2\cos\frac{j\pi}{N+1}$. 
Thus by \eqref{rec} and \eqref{ns}, we have 
\begin{align*}
&	\left|\Phi_{\lambda_N^{(j)}}(|z|)-\langle s_N^{(j)},T_{|z|}s_N^{(j)}\rangle\right|\\
=&\left|\langle s_N^{(j)},\Phi_{\A_{\Gamma_N}}(|z|)s_N^{(j)}\rangle-\langle s_N^{(j)},\Phi_{\A_\Z}(|z|)s_N^{(j)}\rangle\right|\\
\leq & \frac{2}{N+1}\sum_{x,y\in \Gamma_{N}}	\left|	\Phi_{\A_{\Gamma_N}}(|z|)(x,y)-\Phi_{\A_\Z}(|z|)(x,y)\right|\\
\leq & \frac{C_R}{N}
\end{align*}
Noting  that $\rho_z^*=\rho_{-z}$ and 	$\langle s_N^{(j)},\rho_{z}s_N^{(j)}\rangle$ is real, we have 
\begin{equation*}\label{sy}
		\langle s_N^{(j)},\rho_{z}s_N^{(j)}\rangle=\frac{1}{2}\langle s_N^{(j)},(\rho_{z}+\rho_{-z})s_N^{(j)}\rangle=\langle s_N^{(j)},T_{|z|}s_N^{(j)}\rangle.
\end{equation*}
Since  $\Phi_{\lambda_N^{(j)}}(|z|)$ is exactly $\cos\frac{|z|j\pi}{N+1}$ and $R$ is fixed, we finish the proof.
\end{proof}

However, for general $d\geq 2$,  the quantity $ \langle\psi_N^{(j)},\rho_{\z}\psi_N^{(j)}\rangle$ depends on the choice of basis, which  is  non-universal.  For the proof of no correlator universality, we refer to \cite[Section 5.1.2]{MS23} (with the  exponential basis therein replaced by a Dirichlet-adapted sine basis).

\subsection{Further questions}\label{zh}

		In the proof of Theorem \ref{ps}, we crucially use the assumption \eqref{cd} to establish the correspondence of eigenfunctions  between    Dirichlet conditions and periodic conditions. This correspondence does not hold when period lengths exceed two, hence the eigenfunction correspondence might  not be extended  to the general case.  
            \begin{flushleft}\rm 
		\textbf{Question 1:}   
 Does partial quantum ergodicity   hold for $\mathcal{H}_{\Lambda_N}$ in general cases? 
		\end{flushleft}
        
      The  Dirichlet conditions operators  can be viewed as ``small''  perturbations of periodic conditions operators:      $\mathfrak{H}_{\Lambda_N}-\mathcal{H}_{\Lambda_N}$  is  a self-adjoint 0-1 matrix with rank $O(N^{d-1})$.  
    More generally,  consider  self-adjoint operators $\mathcal{W}_{\Lambda_N}$ defined as  
\begin{equation}\label{W}
     \mathcal{W}_{\Lambda_N} = \mathfrak{H}_{\Lambda_N}+\mathcal{P}_{\Lambda_N},
\end{equation}
 where $\mathcal{P}_{\Lambda_N}$ is some   ``small'' self-adjoint perturbation. One can ask  quantum ergodicity questions for $\mathcal{W}_{\Lambda_N}$. 
       Following the methodology in Proof 1, a potentially viable approach would be to compute 	$$a_N^{\infty,\mathcal{W}}:=\lim_{T\to \infty}\frac{1}{T} \int_0^T e^{-{\rm i}t\mathcal{W}_{\Lambda_N}}a_Ne^{{\rm i}t\mathcal{W}_{\Lambda_N}}dt$$ and compare  it with  the periodic condition case $$a_N^{\infty,\mathfrak{H}}:=\lim_{T\to \infty}\frac{1}{T} \int_0^T e^{-{\rm i}t\mathfrak{H}_{\Lambda_N}}a_Ne^{{\rm i}t\mathfrak{H}_{\Lambda_N}}dt.$$ 
		However, since  Bloch-Floquet theory may not be  adapted to $\mathcal{W}_{\Lambda_N}$,  the computation of $a_N^{\infty,\mathcal{W}}$ is more complicated than  $a_N^{\infty,\mathfrak{H}}$.
        
        By \eqref{4}, proving  partial quantum ergodicity for $\mathcal{W}_{\Lambda_N}$ reduces to proving  
		\begin{equation}\label{chabei}
			\lim_{N\to\infty}\frac{1}{N^d}	\|a_N^{\infty,\mathcal{W}}-a_N^{\infty,\mathfrak{H}}\|_{{\rm HS}}^2=0.
		\end{equation} 
        This motivates investigating the stability of  time-averaged  dynamics under self-adjoint perturbations.
         	\begin{flushleft}\rm 
		\textbf{Question 2:}   
Under  what condition on    $\mathcal{P}_{\Lambda_N}$ in \eqref{W} does \eqref{chabei}  hold?
		\end{flushleft}

\vspace{3mm}
\noindent\textbf{Acknowledgments.} \;
Hongyi Cao is supported by NSFC 123B2004.
Shengquan Xiang is partially  supported by NSFC 12301562.

\end{document}